\documentclass{article}
\usepackage{mathtools,amsmath,amssymb,amsthm,graphicx}
\newtheorem{lemma}{Lemma}
\newtheorem{theorem}{Theorem}
\begin{document}
\title{Octonionic presentation for the Lie group $SL(2,{\mathbb O})$
\thanks{Electronic version of an article published as Journal of Algebra and Its Applications, Vol. 13, No. 6 (2014) 1450017 (19 pages) DOI: 10.1142/S0219498814500170 \copyright~Copyright World Scientific Publishing Company http://www.worldscientific.com/worldscinet/jaa}
}
\author{Jean Pierre Veiro
\thanks{Universidad Sim\'on Bol\'{\i}var, Departamento de Matem\'aticas Puras y Aplicadas, Apartado 89000, Caracas 1080-A, Venezuela. Email: jpveiro@usb.ve}
}
\date{}
\maketitle
\begin{abstract}
The purpose of this paper is to provide an octonionic description of the Lie group $SL(2,{\mathbb O})$. The main result states that it can be obtained as a free group generated by invertible and determinant preserving transformations from $\mathfrak{h}_2({\mathbb O})$ onto itself. An interesting characterization is given for the generators of $G_2$. Also, explicit isomorphisms are constructed between the Lie algebras $\mathfrak{sl}(2,{\mathbb K})$, for ${\mathbb K}={\mathbb R}, {\mathbb C}, {\mathbb H}, {\mathbb O}$, and their corresponding Lorentz algebras.
\end{abstract}
\section{Introduction}
In 1984, Sudbery (see~\cite{sudbery}) was able to construct the Lie algebra $\mathfrak{sl}(2,{\mathbb O})$ in such a way that it generalizes the Lie algebras $\mathfrak{sl}(2,{\mathbb K})$ for ${\mathbb K}={\mathbb R}, {\mathbb C}, {\mathbb H}$. If $n$ is the dimension of ${\mathbb K}$, there are Lie algebra isomorphisms between $\mathfrak{sl}(2,{\mathbb K})$ and the Lorentz algebras $\mathfrak{so}(n+1,1)$ for ${\mathbb K}={\mathbb R}, {\mathbb C}, {\mathbb H}, {\mathbb O}$, respectively. They are of most interest given the relation pointed out by Evans (see~\cite{evans}) between the normed division algebras, ${\mathbb R}$, ${\mathbb C}$, ${\mathbb H}$, and ${\mathbb O}$ of dimensions one, two, four, and eight, respectively, and supersymmetric Yang-Mills theories in dimensions three, four, six, and ten. Understanding the octonions is of great importance since they are related to Majorana-Weyl spinors in ten dimensions.
\\

The purpose of this paper is to formulate an octonionic presentation of a Lie group named $SL(2,{\mathbb O})$. The conditions that approve this name are first, it must generalize the definitions of $SL(2,{\mathbb K})$ for ${\mathbb K}={\mathbb R}, {\mathbb C}, {\mathbb H}$ and second, its Lie algebra must be $\mathfrak{sl}(2,{\mathbb O})$. Manogue has provided, with Schray in~\cite{manogue-finite} and with Dray in~\cite{manogue-moebius,manogue-spinor}, a version of $SL(2,{\mathbb O})$ using {\it nested transformations}. Their work has served as inspiration for the approach developed in this paper. One of the benefits of the construction given here is that it shows, at the group level, why in the Lie algebra $\mathfrak{sl}(2,{\mathbb O})$ the factor $\mathfrak{so}(7)$ must be split into commutator maps and $\mathfrak{g}_2$. This separation is also discussed in the work of Manogue {\it et al.} (see~\cite{manogue-finite,manogue-moebius,manogue-spinor}).
\\

The main result states the following:
\begin{quote}
The set of all invertible and determinant preserving transformations $
\phi_M
\:\colon
{\mathfrak h}_2({\mathbb O})\to{\mathfrak h}_2({\mathbb O})
$ defined by
\[
X\longmapsto\frac{1}{2}\left(\left(MX\right)M^\dagger+M\left(XM^\dagger\right)\right)
\]
where $M=\begin{pmatrix}a&b\\c&d\end{pmatrix}$ and $X\in{\mathfrak h}_2({\mathbb O})$, generate a free group with the operation given by composition. It is also a Lie group whose Lie algebra is precisely $\mathfrak{sl}(2,{\mathbb O})$.
\end{quote}
It is in agreement with the construction of $SL(2,{\mathbb O})$ provided by Manogue {\it et al.} in~\cite{manogue-finite,manogue-moebius,manogue-spinor}. The proof uses all the previous results of Section~\ref{groupSL2O}, the isomorphism $\mathfrak{sl}(2,{\mathbb O})\cong\mathfrak{so}(9,1)$ given in Section~\ref{isoOct} and Lemma~\ref{tangenG2}. Lemma~\ref{tangenG2} is of particular interest since it provides infinitesimal generators for $G_2$, in terms of octonionic operations, such that their tangent vectors are precisely the canonical derivations in $\mathfrak{g}_2$.
\\

The outline of this paper is as follows. Section~\ref{prelim} is devoted to $\mathfrak{g}_2$ and $G_2$. Two different bases are given for $\mathfrak{g}_2$; the first being used by Sudbery while the second is used by Manogue {\it et al.} Lemma~\ref{uxu-1} can be found in a book written by Conway \& Smith (see~\cite[page 98]{conway}) although in a different context, with different notation and distinct proof. Lemma~\ref{uxu-1} can also be found in~\cite[pages 3754--3755]{manogue-finite} in terms of rotations. The characterization given in Lemma~\ref{tangenG2} has not been found in the literature and is used in the proof of the main theorem. In Section~\ref{isos}, the isomorphisms between $\mathfrak{sl}(2,{\mathbb K})$ and $\mathfrak{so}(n+1,1)$ for ${\mathbb K}={\mathbb R}, {\mathbb C}, {\mathbb H}, {\mathbb O}$ and $n={\rm dim}({\mathbb K})$, respectively, are shown in detail. Special attention is given to the case of the octonions since it will be a key element in the construction of $SL(2,{\mathbb O})$. The last section contains an octonionic presentation for $SL(2,{\mathbb O})$, some technical lemmas, as well as the main result.
\section{The Octonions and their Automorphisms}\label{prelim}
An algebra is a vector space over the real numbers, $V$, with a bilinear map $m\colon V\times V\to V$ and a nonzero element $1\in V$ such that $m(1,x)=m(x,1)=x$ for all $x\in V$. This bilinear map is called a multiplication in $V$ and will simply be written as $xy$, instead of $m(x,y)$. A normed division algebra is an algebra with a norm that satisfies $||xy||=||x||\,||y||$ for all $x,y\in V$. Normed division algebras are division algebras in the sense that if $xy=0$ then $x=0$ or $y=0$.
\\

A well known theorem due to Hurwitz states that any normed division algebra over the real numbers is isomorphic to either the real numbers, ${\mathbb R}$, the complex numbers, ${\mathbb C}$, the quaternions, ${\mathbb H}$, or the octonions, ${\mathbb O}$. Proofs for this theorem can be found in \cite{conway,schafer}, and a lovely construction of the normed division algebras in \cite{baez}.
\\

These algebras are the first four of a sequence generated by the Cayley-Dickson process. It generalizes the construction of the complex numbers starting from the set of real numbers. A complete description of the Cayley-Dickson algebras is in \cite{schafer,baez}. The quaternions are no longer commutative while the octonions also lack associativity. Since the associator
\[
[a,b,c]=(ab)c-a(bc)
\]
is an alternating trilinear map, any subalgebra of the octonions generated by only two elements is associative. The general statement is known as Artin's Theorem, see \cite[page 29]{schafer}.
\\

Given a normed division algebra ${\mathbb K}$, its group of proper\footnote{Proper automorphisms refer to those satisfying $\phi(xy)=\phi(x)\,\phi(y)$ for all $x,y\in{\mathbb K}$.} automorphisms is denoted by ${\rm Aut}({\mathbb K})$. Therefore, ${\rm Aut}({\mathbb R})\cong\{{\rm id}\}$ is the trivial group, ${\rm Aut}({\mathbb C})\cong{\mathbb Z}_2$ is given by the identity and complex conjugation, ${\rm Aut}({\mathbb H})\cong SO(3)$ are the orientation preserving rotations over the pure imaginary quaternions, and ${\rm Aut}({\mathbb O})\cong G_2$ is the 14 dimensional exceptional Lie group.
\\

Multiplication of octonions follows from the distributive law and a multiplication table for the basis elements illustrated in the following Fano plane.
\begin{figure}[th]
\begin{center}
\vspace*{-.3cm}
\includegraphics[width=4.8cm]{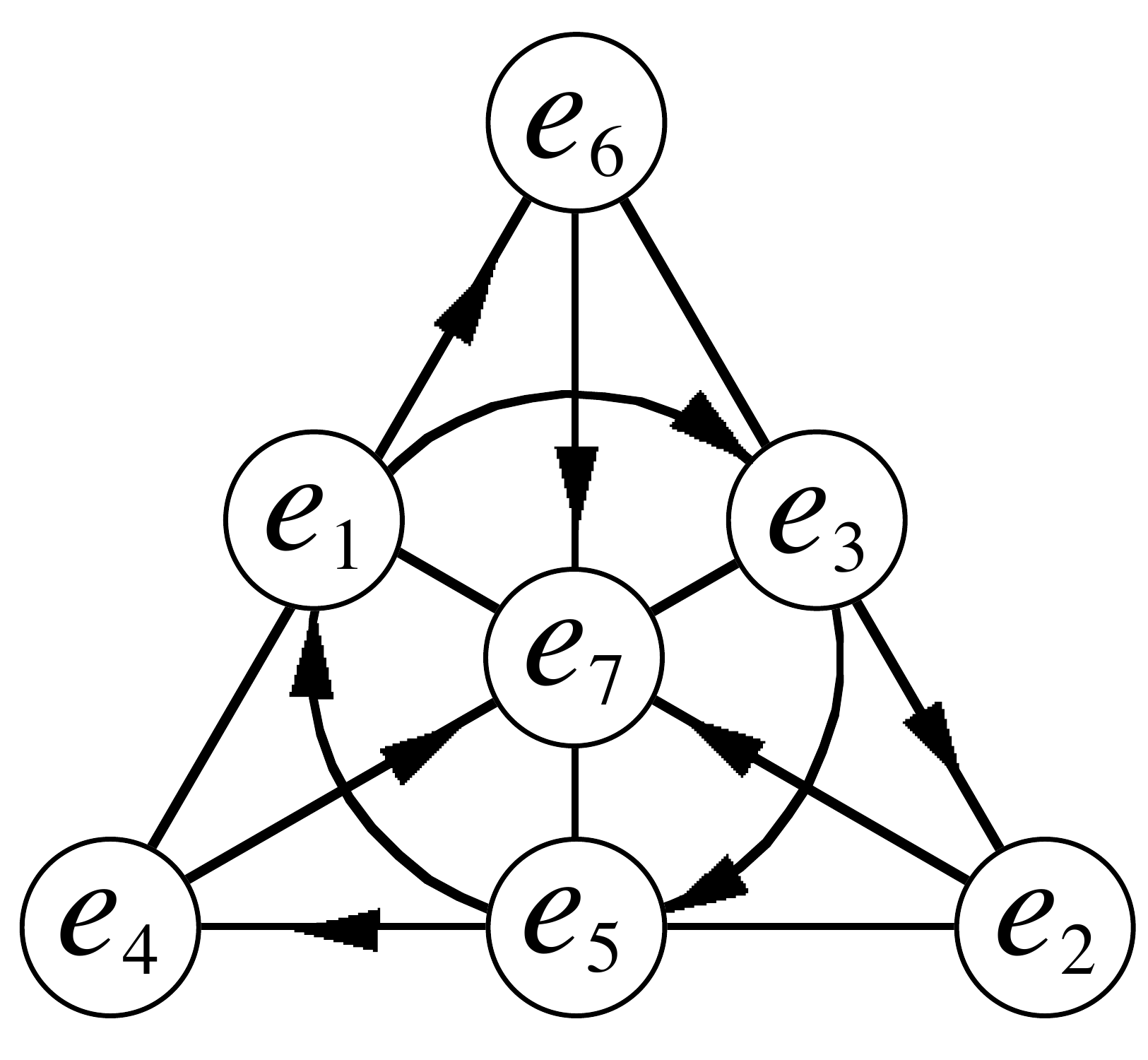}
\vspace*{-.35cm}
\caption{Fano Plane representing the multiplication table for the octonions.}
\vspace*{-.3cm}
\end{center}
\end{figure}
The result of the product of two elements is the only other element that shares the line passing through the first two, and the sign is given by the arrows. For example, $e_3\,e_5=e_1$ while $e_1\,e_4=-e_6$. The basis element $e_0=1$ is left outside the diagram because it commutes and associates with all other elements.
\\

The octonions satisfy the Moufang identities; for any octonions $x,y,z\in{\mathbb O}$, the following relations hold.
\begin{eqnarray*}
z(x(zy))&=&(zxz)y\\
((xz)y)z&=&x(zyz)\\
(zx)(yz)&=&z(xy)z
\end{eqnarray*}
As a consequence of the Moufang identities it is also true that
\[
(xy)z=(xz^{-1})(zyz)\qquad\mbox{and}\qquad z(xy)=(zxz)(z^{-1}y)
\]
whenever $z\neq0$.
\\

Besides alternativity and power associativity, the following properties hold.
\begin{eqnarray*}
&\left[\bar{x},y\right]=-[x,y]=\overline{[x,y]}&\\
&\left[\bar{x},y,z\right]=-[x,y,z]=\overline{[x,y,z]}&
\end{eqnarray*}
Another interesting relation is
\[
6\,[x,y,z]=[[x,y],z]+[[y,z],x]+[[z,x],y].
\]
It shows why the imaginary part of the quaternions, with multiplication given by the commutator, form a Lie algebra and the same done with the octonions fails.
\\

The derivation algebra over the octonions is $\mathfrak{Der}({\mathbb O})\cong\mathfrak{g}_2$; it is the Lie algebra tangent to the group of automorphisms ${\rm Aut}({\mathbb O})\cong G_2$. Two bases for $\mathfrak{g}_2$ will be exhibited. Consider the following derivations
\begin{eqnarray}
D_{a,b}(x)&=&\big([L_a,L_b]+[L_a,R_b]+[R_a,R_b]\big)(x)\nonumber\\
&=&\frac{1}{2}\big([[a,x],b]+[a,[b,x]]+[[a,b],x]\big)\label{Dab}\\
&=&[[a,b],x]-3[a,b,x]\nonumber
\end{eqnarray}
defined for any pair of octonions $a$ and $b$. They satisfy the Generalized Jacobi Identity,
\[
D_{a,b}\left(D_{c,d}\right)=D_{D_{a,b}(c),d}+D_{c,D_{a,b}(d)}+D_{c,d}\left(D_{a,b}\right)
\]
and the following linear dependence relations
\[
D_{e_i,e_j}+D_{e_m,e_n}+D_{e_r,e_s}=0
\]
where the indices are all different and such that $e_ie_j=e_me_n=e_re_s=e_k$ for $k=1,\ldots,7$, ruling out seven of the possible 21 transformations and leaving a set of 14 elements for the basis. Another way to construct a basis for $\mathfrak{g}_2$ is considering transformations of the form
\[
R_{ij}(x)={\rm Re}(xe_j)\,e_i-{\rm Re}(xe_i)\,e_j.
\]
These can be understood as tangent vectors to rotations in the plane determined by $e_i$ and $e_j$. Nonetheless, these transformations are not derivations over the octonions. The reason is that, at the group level, in order to obtain an automorphism it is necessary to rotate in certain two planes simultaneously. Consider the following transformations in $\mathfrak{g}_2$,
\[
F^k_{ij}=R_{ia}+R_{jb}
\]
where the indices $a$ and $b$ are such that $e_ie_a=e_k=-e_je_b$. Each fixed $k\in\{1,\ldots,7\}$ determines six transformations that span a subspace in $\mathfrak{g}_2$ of dimension two. This leads to 14 transformations that are linearly independent in $\mathfrak{g}_2$ and, thus, a basis. For example, the set
\[
\left\{
F^{1}_{24},\,F^{1}_{25},\,F^{2}_{51},\,F^{2}_{53},\,F^{3}_{21},\,F^{3}_{27},\,F^{4}_{23},\,F^{4}_{13},\,F^{5}_{17},\,F^{5}_{12},\,F^{6}_{31},\,F^{6}_{71},\,F^{7}_{36},\,F^{7}_{16}
\right\}
\]
is a basis for $\mathfrak{g}_2$.
\\

Each basis has its own set of infinitesimal generators. For the latter we shall visualize $\mathfrak{g}_2$ lying in $\mathfrak{so}(7)\subset\mathfrak{so}(8)$. Noticing that for any vector belonging to this basis we have
\[
\big(F^k_{ij}\big)^3=-F^k_{ij},\quad\big(F^k_{ij}\big)^4=-\big(F^k_{ij}\big)^2,\quad\big(F^k_{ij}\big)^5=F^k_{ij},\quad\ldots
\]
and so on, the exponential map determines infinitesimal generators for a realization of $G_2$ inside $SO(7)\subset SO(8)$ by
\[
{\rm exp}\big(t\,F^k_{ij}\big)={\rm Id}+\sin(t)\big(F^k_{ij}\big)+(1-\cos(t)){\big(F^k_{ij}\big)}^2.
\]

The former basis is related to a trickier characterization of generators for $G_2$. The following Lemma allows us to determine when does an automorphism of octonions adopt the same type of expression as those for the quaternions.
\begin{lemma}\label{uxu-1}
If $u$ is a nonzero octonion then the transformation
\[
x\longmapsto u\,x\,u^{-1}
\]
is an automorphism of octonions if, and only if, $u^3$ is real.
\end{lemma}
\begin{proof}
Artin's Theorem will constantly be used throughout the proof. The Moufang identity $(xyx)(x^{-1}z)=x(yz)$ shows that
\[
(u\,a\,u^{-1})(u\,b\,u^{-1})=u^{-1}\left((u^2a)(bu^{-1})\right)
\]
for any given octonions $a,b,u\in{\mathbb O}$. Therefore, the transformation that sends $x$ into $u\,x\,u^{-1}$ is an automorphism for the octonions if, and only if,
\[
u(ab)u^{-1}=u^{-1}\left((u^2a)(bu^{-1})\right)
\]
for all $a,b\in{\mathbb O}$. Multiplying on the right by $u$ and on the left by $u^2$ yields
\[
u^3(ab)=u\left((u^2a)(bu^{-1})\right)u
\]
while the Moufang identity $x(yz)x=(xy)(yz)$ rewrites the expression's right hand side as $(u^3a)b$. Thus, the given transformation is an automorphism for the octonions if, and only if,
\[
\left[u^3,a,b\right]=0
\]
for all $a,b\in{\mathbb O}$. Using the next Lemma, this last condition holds if, and only if, $u^3$ is real.
\end{proof}
\begin{lemma}
$[x,a,b]=0$ for all $a,b\in{\mathbb O}$ if, and only if, $x\in{\mathbb R}$.
\end{lemma}
\begin{proof}
First suppose that $[x,a,b]=0$ for all $a,b\in{\mathbb O}$. Given a basis for the octonions, $\{e_0,e_1,\ldots,e_7\}$ with $e_0=1$, choose different indices $i,j,n,m\in\{1,\ldots,7\}$ such that $e_ie_j=e_ne_m=e_k$. Then, $e_ie_n=e_l\neq e_k$. Since $[x,a,b]=0$ for all $a,b\in{\mathbb O}$, then particularly
\begin{eqnarray*}
\left[x,e_i,e_j\right]=0&\Rightarrow &x=x_0 e_0+x_i e_i+x_j e_j+x_k e_k,\\
\left[x,e_n,e_m\right]=0&\Rightarrow &x=x_0 e_0+x_n e_n+x_m e_m+x_k e_k,\\
\left[x,e_i,e_n\right]=0&\Rightarrow &x=x_0 e_0+x_i e_i+x_n e_n+x_l e_l.
\end{eqnarray*}
For $x$ to satisfy $\left[x,e_i,e_j\right]=0$ and $\left[x,e_n,e_m\right]=0$ simultaneously, $x$ must be of the form $x=x_0 e_0+x_k e_k$. That is, $x_i=x_j=x_n=x_m=0$. If $x$ must also satisfy $\left[x,e_i,e_n\right]=0$, since $e_l\neq x_k$, then necessarily $x_l=0$ and $x_k=0$. Thus $x$ can only be real. Conversely, if $x\in{\mathbb R}$, it is easily seen that $[x,a,b]=0$ for all $a,b\in{\mathbb O}$.
\end{proof}

It is interesting to interpret $u^3\in{\mathbb R}$ geometrically. Considering $u=u_0 e_0+\vec{u}$, where $\vec{u}=u_1 e_1+\cdots+u_7 e_7$, then
\[
u^3=u_0^3-3u_0||\vec{u}||^2+\left(3u_0^2-||\vec{u}||^2\right)\vec{u}
\]
and $u^3\in{\mathbb R}$ if either $\vec{v}=0$ or $3u_0^2-||\vec{u}||^2=0$. Representing the real part of the octonions as a straight line and the imaginary part as a plane that is perpendicular to such line, then the octonions $u=u_0 e_0+\vec{u}$ such that $3u_0^2-||\vec{u}||^2=0$ can be identified with a cone passing through the origin whose axis is the real line.
\\

Being interested in transformations of the form $x\mapsto u\,x\,u^{-1}$, $u$ cannot be equal to zero. Also, if $w=\lambda\,u$ with $\lambda\in{\mathbb R}$ then $w\,x\,w^{-1}=u\,x\,u^{-1}$. Therefore, it is possible to quotient the union of the octonions that are real with the set of octonions $u$ such that $3u_0^2=||\vec{u}||^2$ by the relation $u\sim w$ if $w=\lambda\, u$ for some real number $\lambda$ to obtain representatives of these transformations. The entire real line can be identified with the identity transformation. Unfortunately this element is disconnected from the rest. In order to construct nontrivial curves passing through the identity in $G_2$, a composition of at least two transformations must be considered. Due to nonasociativity, composition of these transformations cannot be rewritten as just one transformation of the same kind. More precisely, if $u$ and $v$ are two octonions such that $u^3\in{\mathbb R}$ and $v^3\in{\mathbb R}$ then $v\,(u\,x\,u^{-1})\,v^{-1}=w\,x\,w^{-1}$, for some octonion $w$ that also satisfies $w^3\in{\mathbb R}$, only when $[v,u,x]=0$ for all $x\in{\mathbb O}$.
\\

Consider the following notation. Given two imaginary octonions $a,b\in{\mathbb O}$ (with ${\rm Re}(a)={\rm Re}(b)=0$) that anti-commute, define
\begin{equation}\label{u(t)}
u(t)=\frac{1}{2}+\frac{\sqrt{3}}{2\left|\left|a+\frac{4}{3}||a||^2\,t\,b\right|\right|}\left(a+\frac{4}{3}||a||^2\,t\,b\right).
\end{equation}
Notice that $||u(t)||=1$ and $(u(t))^3\in{\mathbb R}$ for all $t\in{\mathbb R}$. Hence, both the maps $x\mapsto u(t)\,x\,(u(t))^{-1}$ and $x\mapsto (u(t))^{-1}\,x\,u(t)$ are in $G_2$ for any value of $t$. Since $u$ depends on the choice of $a$ and $b$, define $G_{a,b}^t$ as the transformation that sends $x$ to
\begin{equation}\label{gut}
(u(t))^{-1}\Big(u(0)\Big(u(t)\Big((u(0))^{-1}\,x\,u(0)\Big)(u(t))^{-1}\Big)(u(0))^{-1}\Big)u(t).
\end{equation}
Clearly, $G_{a,b}^t\in G_2$, for all $t\in{\mathbb R}$, since it is the composition of transformations in $G_2$. Also, $G_{a,b}^0$ is always the identity element. In fact, $G_{a,b}^t$ is a curve in $G_2$ that passes through the the identity at $t=0$.
\begin{lemma}\label{tangenG2}
Given two nonzero octonions, $a,b\in{\mathbb O}$ such that ${\rm Re}(a)={\rm Re}(b)=0$ and $ab=-ba$, consider the curve $G_{a,b}^t$ defined by~\eqref{gut} where $u(t)$ is as in~\eqref{u(t)}. Then the tangent vector to this curve at the identity is $D_{a,b}$, as defined in~\eqref{Dab}.
\end{lemma}
\begin{proof}
The map that assigns $t\mapsto u(t)$ is differentiable. Also, since $||u(t)||=1$ for all $t\in{\mathbb R}$, then $(u(t))^{-1}=\overline{u(t)}$. When $t=0$ only two octonionic directions are involved in~\eqref{gut} and therefore, Artin's Theorem may be applied to reduce this expression to the identity transformation. The set of transformations $G_{a,b}^t$, for $t$ belonging to any open interval around zero, define differentiable curves in $G_2$. The following calculations will reveal its tangent vector. For simplicity, allow $\dot{u}$ and $u_0$ to denote ${\left.\frac{{\rm d}}{{\rm d}t}u(t)\right|}_{t=0}$ and $u(0)$ respectively.
\begin{eqnarray*}
{\left.\frac{{\rm d}}{{\rm d}t}G_{a,b}^t(x)\right|}_{t=0}
&=&-\big(u_0^{-1}\,\dot{u}\,u_0^{-1}\big)(u_0\,x)+\Big(\dot{u}\big(u_0^{-1}\,x\,u_0\big)\Big)u_0^{-1}\\
&&-(x\,u_0)\big(u_0^{-1}\,\dot{u}\,u_0^{-1}\big)+u_0^{-1}\Big(\big(u_0\,x\,u_0^{-1}\big)\dot{u}\Big)
\end{eqnarray*}
Using the Moufang identities and properties related with the associator, the previous expression is equal to
\begin{equation}\label{dGu}
\Big[\big[\dot{u},u_0^{-1}\big],x\Big]
+\Big[\big(u_0-u_0^{-1}\big)\big(\dot{u}\,u_0^{-1}\big)+\big(u_0^{-1}\,\dot{u}\big(u_0-u_0^{-1}\big),u_0^{-1},x\Big].
\end{equation}
From~\eqref{u(t)}, explicit calculations determine
\[u_0=u(0)=\frac{1}{2}+\frac{\sqrt{3}}{2||a||}\,a
\qquad{\rm and}\qquad
\dot{u}=-\frac{2{\rm Re}(a\bar{b})}{\sqrt{3}||a||}+\frac{2||a||}{\sqrt{3}}\,b.\]
Replacing these values of $u_0$ and $\dot{u}$ in~\eqref{dGu}, simplifying, and using the fact that $aba=-a^2b=||a||^2b$, since $ab+ba=0$, yields
\[
{\left.\frac{{\rm d}}{{\rm d}t}G_{a,b}^t(x)\right|}_{t=0}=[[a,b],x]-3[a,b,x]=D_{a,b}(x)
\]
for all $x\in{\mathbb O}$. Thus the tangent vector to the curve $G_{a,b}^t$, when passing through the identity, is the derivation $D_{a,b}$ for the given $a,b\in{\mathbb O}$.
\end{proof}
If $i,j\in\{1,\ldots,7\}$ with $i\neq j$ then $a=e_i$ and $b=e_j$ satisfy the conditions in Lemma~\ref{tangenG2}. This allows us to exhibit the basis $\{D_{e_i,e_j}\}$ for $\mathfrak{g}_2$ as the set of vectors tangent to the curves $G_{e_i,e_j}^t$ at $t=0$.
\\

Noticing that Eq.~\eqref{u(t)} can be rewritten as $u=\frac{1}{2}+\frac{\sqrt{3}}{2}\hat{r}$ with $\hat{r}\in S^6$, shows how this equation is equivalent to Eq.~(31) in~\cite{manogue-finite}. Manogue \& Schray characterize the automorphisms of the form $x\mapsto u\,x\,u^{-1}$ by expressing the octonion $u$ in polar coordinates. They conclude that the set of transformations that map $x$ to $u\,x\,\bar{u}$, with $u=\frac{1}{2}+\frac{\sqrt{3}}{2}\hat{r}$ is a generating set for $G_2$. In order to obtain curves that pass through the identity they consider maps of the form $\phi^{(2)}_{(i,j,\theta|\frac{\pi}{3})}$; each of them being composition of two maps in the generating set, where one depends on the parameter $\theta$ and the other is fixed so that the curve passes through the identity for $\theta=0$. These curves are related to the ones defined by~\eqref{gut}, namely $G^t_{a,b}$, in the following way. Choosing $a=e_i$, $b=e_j$, and writting $t=\frac{3}{4}{\rm tan}(\theta)$ for $\theta\in(-\frac{\pi}{2},\frac{\pi}{2})$, then
\[
{G^t_{a,b}=\phi^{(2)}_{(i,j,\pi+\theta|\frac{\pi}{3})}\circ\phi^{(2)}_{(i,j,\theta|\frac{\pi}{3})}}_{.}
\]
They later consider the space spanned by the vectors $\left\{\left.\frac{{\rm d}}{{\rm d}\theta}\phi^{(2)}_{(i,j,\theta|\frac{\pi}{3})}\right|_{\theta=0}\right\}$ with $i\neq j$ and going through all the indices of the imaginary part of ${\mathbb O}$. Using the matrix representation of $G_2$, Manogue \& Schray are able to extract 14 linearly independent vectors and thus obtain a basis for $\mathfrak{g}_2$. Unfortunately, these vectors are not easily expressed in terms of the derivations $D_{a,b}$ nor of the transformations $F^k_{i\,j}$ listed above.
\section{The Lorentz algebras and their isomorphic special linear algebras}\label{isos}
Consider ${\mathbb R}^{n+1}$ with the following quadratic form
\[
Q(x)=x_0^2-x_1^2-x_2^2-\cdots-x_n^2
\]
where $x=(x_0,x_1,\ldots,x_n)\in{\mathbb R}^{n+1}$. Using the polarization identity, the symmetric bilinear form $f$, associated to the quadratic form $Q$, can be recovered. Then,
\begin{eqnarray*}
f(x,x^\prime)&=&\frac{1}{4}\big(Q(x+x^\prime)-Q(x-x^\prime)\big)\\
&=&x_0x_0^\prime-x_1x_1^\prime-x_2x_2^\prime-\cdots-x_n^\prime
\end{eqnarray*}
for any $x,x^\prime\in{\mathbb R}^{n+1}$. Notice that $f$ is non-degenerate. Such bilinear form is called the Lorentz bilinear form. It can be expressed
\[
f(x,x^\prime)=x\,\eta\,x^\prime
\]
where $x$ is written as a row, $x^\prime$ as a column, and $\eta$ is the diagonal $(n+1)\times(n+1)$ matrix $\eta={\rm diag}(1,-1,-1,\ldots,-1)$.
\\

Notice that $f(x,x^\prime)=\langle\eta\,x,x^\prime\rangle$. The group of invertible linear transformations over ${\mathbb R}^{n+1}$ that preserve the Lorentz bilinear form is called the Lorentz group of ${\mathbb R}^{n+1}$, and will be denoted by
\[
O(n,1)=\left\{\Lambda\in GL(n+1,{\mathbb R})\;:\;f(\Lambda\,x,\Lambda\,x^\prime)=f(x,x^\prime),\;\forall\,x,x^\prime\in{\mathbb R}^{n+1}\right\}
\]
where $GL(n+1,{\mathbb R})$ is the general linear group of $(n+1)\times(n+1)$ matrices with entries in ${\mathbb R}$. Elements in the Lorentz group, $\Lambda\in O(n,1)$, can be characterized by the following condition
\[
\Lambda^T\,\eta\,\Lambda=\eta.
\]
The Lorentz group is obviously a Lie group. Its connected component containing the identity is denoted by $SO(n,1)$ and is called the Orthochronous Lorentz group.
\\

The Lorentz algebra, $\mathfrak{so}(n,1)$, tangent to $SO(n,1)$ is
\[
\mathfrak{so}(n,1)=\left\{\omega\in{\rm M}_{n+1}({\mathbb R})\;:\;\eta\,\omega=-(\eta\,\omega)^T\right\}
\]
where ${\rm M}_{n+1}({\mathbb R})$ stands for $(n+1)\times(n+1)$ matrices with entries in ${\mathbb R}$. The Lie bracket in $\mathfrak{so}(n,1)$ is given by $[\omega,\omega^\prime]=\omega\,\omega^\prime-\omega^\prime\,\omega$.
\subsection{The isomorphism $\mathfrak{sl}(2,{\mathbb R})\cong\mathfrak{so}(2,1)$}
The Lie algebra $\mathfrak{sl}(2,{\mathbb R})$ consists of $2\times2$ matrices with entries in ${\mathbb R}$ and trace equal to zero. Define $\varphi\colon\mathfrak{sl}(2,{\mathbb R})\to\mathfrak{so}(2,1)$ by
\[
\varphi
\begin{pmatrix*}[r]
a&b\\
c&-a
\end{pmatrix*}
=
\begin{pmatrix}
0&b+c&2a\\
b+c&0&c-b\\
2a&b-c&0
\end{pmatrix}_.
\]
It is easy to verify that $\varphi$ is a Lie algebra isomorphism. The next subsection will show why this expression for $\varphi$ is, by all means, natural.
\subsection{The isomorphism $\mathfrak{sl}(2,{\mathbb C})\cong\mathfrak{so}(3,1)$}
Let $x\in{\mathbb R}^4$ and consider the following identification
\[
x=(x_0,x_1,x_2,x_3)\:\longleftrightarrow\: X=
\begin{pmatrix}
x_0+x_3&x_1+x_2\,i\\
x_1-x_2\,i&x_0-x_3
\end{pmatrix}
\]
between vectors in ${\mathbb R}^4$ and $2\times2$ hermitian matrices with entries in ${\mathbb C}$, denoted by $\mathfrak{h}_2({\mathbb C})$. Notice that
\[
{\rm det}(X)=Q(x)=f(x,x),
\]
thus it is of interest to preserve the determinant of the matrix $X$.
\\

Let $M\in M_2({\mathbb C})$, a $2\times2$ matrix with entries in ${\mathbb C}$. Consider the Hermitian action of $M$ on $X$ as shown in the following map
\[
X\longmapsto MXM^\dagger
\]
where $\dagger$ represents the transpose of the complex conjugate matrix. Writing $X^\prime=MXM^\dagger$,
\[
\det(X^\prime)=\det(MXM^\dagger)=\det(MM^\dagger)\det(X)=\|\det(M)\|^2\det(X).
\]
Noticing that $\|\det(e^{i\theta}M)\|^2=\|\det M\|^2$ for any value of $\theta$ and that both $e^{i\theta}M$ and $M$ produce the same transformation, it can be concluded that imposing $M$ to have determinant equal to one is sufficient to characterize the determinant preserving Hermitian action. That is, $M\in SL(2,{\mathbb C})$, the Lie group of matrices $A\in M_2({\mathbb C})$ such that ${\rm det}(A)=1$.
\\

At the Lie algebra level, $\mathfrak{sl}(2,{\mathbb C})$ is the set of traceless $2\times2$ matrices over the complex numbers, ${\rm Tr}(N)=0$, with the Lie bracket given by matrix commutation. Wishing to send tangent vectors, $\vec{X}$, with base point $X\in\mathfrak{h}_2({\mathbb C})$ onto tangent vectors with base point $X^\prime=MXM^\dagger$, it is necessary for $N\in\mathfrak{sl}(2,{\mathbb C})$ to transform as follows,
\[
\vec{X}\longmapsto N\vec{X}+\vec{X}N^\dagger=\vec{X}^\prime.
\]
Writing $N=\begin{pmatrix}\alpha_0+\alpha_1i&\beta_0+\beta_1i\\ \gamma_0+\gamma_1i&\alpha_0-\alpha_1i\end{pmatrix}$ and $\vec{X}$ the same way as $X$, then the above transformation corresponds to the following
\[
\begin{pmatrix}
0&\beta_0+\gamma_0&\beta_1-\gamma_1&2\alpha_0\\
\gamma_0+\beta_0&0&-2\alpha_1&\gamma_0-\beta_0\\
\beta_1-\gamma_1&2\alpha_1&0&-\beta_1-\gamma_1\\
2\alpha_0&\beta_0-\gamma_0&\beta_1+\gamma_1&0
\end{pmatrix}
\begin{pmatrix}
x_0\\
x_1\\
x_2\\
x_3
\end{pmatrix}
=
\begin{pmatrix}
x_0^\prime\\
x_1^\prime\\
x_2^\prime\\
x_3^\prime
\end{pmatrix}
\]
that is actually a transformation belonging to $\mathfrak{so}(3,1)$. Therefor, define $\varphi\colon\mathfrak{sl}(2,{\mathbb C})\to\mathfrak{so}(3,1)$ by
\[
\varphi
\begin{pmatrix}
\alpha_0+\alpha_1i&\beta_0+\beta_1i\\
\gamma_0+\gamma_1i&-\alpha_0-\alpha_1i
\end{pmatrix}
={
\begin{pmatrix}
0&\beta_0+\gamma_0&\beta_1-\gamma_1&2\alpha_0\\
\gamma_0+\beta_0&0&-2\alpha_1&\gamma_0-\beta_0\\
\beta_1-\gamma_1&2\alpha_1&0&-\beta_1-\gamma_1\\
2\alpha_0&\beta_0-\gamma_0&\beta_1+\gamma_1&0
\end{pmatrix}
}_.
\]
It is easy to see that $\varphi$ is, in fact, a Lie algebra isomorphism.
\subsection{The isomorphism $\mathfrak{sl}(2,{\mathbb H})\cong\mathfrak{so}(5,1)$}
Repeating the ideas from the previous section, $\mathfrak{sl}(2,{\mathbb H})$ cannot be defined using as a condition that ${\rm Tr}(A)=0$ if $A\in M_2({\mathbb H})$. The reason is that if the Lie bracket is taken as the matrix commutator, then ${\rm Tr}([A,B])$ is not necessarily equal to zero even when both ${\rm Tr}(A)$ and ${\rm Tr}(B)$ are equal to zero. Weakening this condition to just ${\rm Re}({\rm Tr}(A))=0$, that is, the trace having its real part equal to zero, does close the algebra under matrix commutation. It will be shown that this is a natural definition for $\mathfrak{sl}(2,{\mathbb H})$.
\\

Let $x\in{\mathbb R}^6$ and consider the following identification
\[
x\!\!=\!\!(x_0,x_1,\ldots,x_5)\:\longleftrightarrow\: X\!\!=\!\!
\begin{pmatrix}
x_0+x_5&x_1+x_2\,i+x_3\,j+x_4\,k\\
x_1-x_2\,i-x_3\,j-x_4\,k&x_0-x_5
\end{pmatrix}
\]
between points in ${\mathbb R}^6$ and $2\times2$ hermitian matrices with entries in ${\mathbb H}$, denoted by $\mathfrak{h}_2({\mathbb H})$. Notice the determinant is well defined in $\mathfrak{h}_2({\mathbb H})$ and
\[
{\rm det}(X)=Q(x)=f(x,x^\prime).
\]
As before, consider the Hermitian action of $M=\begin{pmatrix}a&b\\ c&d\end{pmatrix}\in M_2({\mathbb H})$ on $X$ by
\[
X\longmapsto MXM^\dagger=X^\prime.
\]
Fortunately, despite quaternions are non commutative, the following equality remains true
\begin{eqnarray*}
{\rm det}(X^\prime)
&=&{\rm det}(MXM^\dagger)\\
&=&\left(\|a\|^2\|d\|^2+\|b\|^2\|c\|^2-2\,{\rm Re}\left(a\bar{c}d\bar{b}\right)\right){\rm det}(X)\\
&=&{\rm det}(MM^\dagger)\,{\rm det}(X).
\end{eqnarray*}
Notice the determinant of $MM^\dagger$ is well defined because $MM^\dagger\in\mathfrak{h}_2({\mathbb H})$. Hence, define
\[
SL(2,{\mathbb H})=\left\{A\in M_2({\mathbb H})\;:\;{\rm det}(AA^\dagger)=1\right\}
\]
which is a Lie subgroup of $M_2({\mathbb H})$. This is a good moment to remark that changing the condition ${\rm det}(AA^\prime)=1$ for ${\rm Re}({\rm det}(A))=1$ fails to determine a group. The group's Lie algebra is, in fact,
\[
\mathfrak{sl}(2,{\mathbb H})=\left\{N\in M_2({\mathbb H})\;:\;{\rm Re}({\rm Tr}(N))=0\right\}
\]
with the Lie bracket given by matrix commutation, as wished for at the beginning of this subsection.
\\

As in the previous subsection, matrices $N\in\mathfrak{sl}(2,{\mathbb H})$ transform tangent vectors as follows
\[
\vec{X}\longmapsto N\vec{X}+\vec{X}N^\dagger.
\]
Writing $N=\begin{pmatrix}\alpha&\beta\\ \gamma&\delta\end{pmatrix}$ with $\alpha=\alpha_0\alpha_1\,i+\alpha_2\,j+\alpha_3\,k$ and $\beta$, $\gamma$, and $\delta$ the same way, the above transformation can be associated with a matrix in $\mathfrak{so}(5,1)$, namely
\[
\varphi(N)
=
{
\left(
\begin{smallmatrix}
0 & \beta_0+\gamma_0 & \beta_1-\gamma_1 & \beta_2-\gamma_2 & \beta_3-\gamma_3 & 2\alpha_0 \\
\beta_0+\gamma_0 & 0 & \delta_1-\alpha_1 & \delta_2-\alpha_2 & \delta_3-\alpha_3 & \gamma_0-\beta_0 \\
\beta_1-\gamma_1 & \alpha_1-\delta_1 & 0 & -(\alpha_3+\delta_3) & \alpha_2+\delta_2 & -(\beta_1+\gamma_1) \\
\beta_2-\gamma_2 & \alpha_2-\delta_2 & \alpha_3+\delta_3 & 0 & -(\alpha_1+\delta_1) & -(\beta_2+\gamma_2) \\
\beta_3-\gamma_3 & \alpha_3-\delta_3 & -(\alpha_2+\delta_2) & \alpha_1+\delta_1 & 0 & -(\beta_3+\gamma_3) \\
2\alpha_0 & \beta_0-\gamma_0 & \beta_1+\gamma_1 & \beta_2+\gamma_2 & \beta_3+\gamma_3 & 0
\end{smallmatrix}
\right)
}_{.}
\]
The map $\varphi\colon\mathfrak{sl}(2,{\mathbb H})\to\mathfrak{so}(5,1)$ is a Lie algebra isomorphism.
\subsection{The isomorphism $\mathfrak{sl}(2,{\mathbb O})\cong\mathfrak{so}(9,1)$}\label{isoOct}
This is by far the most interesting of all four isomorphisms shown in this section. When the normed division algebra is the octonions, there is an additional complication compared to the previous subsection. Working with the real part of the trace equal to zero does not close the algebra under multiplication given by matrix commutation, moreover, the Jacobi identity is no longer valid. Thus, a detour must be taken in order to define a Lie algebra suitably called $\mathfrak{sl}(2,{\mathbb O})$. Sudbery's work \cite{sudbery} will be followed very closely to define an appropriate vector space.
\\

A Jordan algebra, $J$, is a commutative algebra such that
\[
x\circ(y\circ x^2)=(x\circ y)\circ x^2
\]
for all $x,y\in J$. Consider the Lie algebra generated by all the Jordan multiplications $L_x(y)=x\circ  y=R_x(y)$. This algebra, named by Sudbery as the structure algebra of $J$, is $J$'s multiplication algebra (see \cite{schafer}). Since Sudbery's work is being followed here, denote this algebra by $\mathfrak{Str}(J)$.
\\

When dealing with a semi-simple Jordan algebra, the structure algebra can be expressed as
\[
\mathfrak{Str}(J)=\mathfrak{Der}(J)\oplus L(J)
\]
where $\mathfrak{Der}(J)$ are the derivations over $J$, $L(J)$ are the Jordan multiplications, and $\oplus$ refers to the direct sum of vector spaces. The algebra of interest for what follows is the reduced structure algebra, denoted by $\mathfrak{Str}^\prime(J)$, which is obtained from the structure algebra after canceling the multiples of the identity in $L(J)$.
\\

In the previous subsections it has been noticed that $\mathfrak{h}_2({\mathbb K})$ is ideal for describing vectors in light-cone coordinates. Best of all, $\mathfrak{h}_2({\mathbb K})$ with the product given by
\[
X\circ Y=\frac{1}{2}(XY+YX),
\]
for all $X,Y\in\mathfrak{h}_2({\mathbb K})$, is a semi-simple Jordan algebra.
\\

The fundamental ingredient when studying $\mathfrak{Str}^\prime(J)$ is $\mathfrak{Der}(J)$. Notice first that if $X,Y\in\mathfrak{h}_2({\mathbb K})$ then
\[
[A,\{X,Y\}]=\{[A,X],Y\}+\{X,[A,Y]\}
\]
holds for any $2\times2$ matrix $A$, where $\{X,Y\}=2(X\circ Y)$. If it is wished that $[A,\{X,Y\}]\in\mathfrak{h}_2({\mathbb K})$ then it is necessary for $A$ to be an anti-hermitian matrix. That is, $A\in\mathfrak{a}_2({\mathbb K})=\{M\in M_2({\mathbb K})\;:\;M^\dagger=-M\}$. Therefore, ${\rm ad}\,A(X)=[A,X]$ is a Jordan derivation over $\mathfrak{h}_2({\mathbb K})$. Furthermore, derivations over ${\mathbb K}$ are derivations of $\mathfrak{h}_2({\mathbb K})$ considering their action over the entries of the matrix.
\\

Using the prime sign to denote traceless matrices or pure imaginary elements of the normed division algebra, $M_2^\prime({\mathbb K})=\{M\in M_2({\mathbb K})\;:\;{\rm Tr}(M)=0\}$ and ${\mathbb K}^\prime=\{x\in{\mathbb K}\;:\;{\rm Re}(x)=0\}$, notice that $M_2^\prime({\mathbb K})=\mathfrak{a}_2^\prime({\mathbb K})+\mathfrak{h}_2^\prime({\mathbb K})$. Writing $\mathfrak{a}_2({\mathbb K})=\mathfrak{a}_2^\prime({\mathbb K})\oplus{\mathbb K}^\prime I$, where $I$ is the $2\times2$ identity matrix, it can be verified that
\[
{\rm ad}\,\mathfrak{a}_2({\mathbb K})={\rm ad}\,\mathfrak{a}_2^\prime({\mathbb K})+{\rm ad}\,({\mathbb K}^\prime I)={\rm ad}\,\mathfrak{a}_2^\prime({\mathbb K})+C({\mathbb K}^\prime),
\]
where $C({\mathbb K}^\prime)=\{C_a=L_a-R_a\;:\;a\in{\mathbb K}^\prime\}$. Hence, the following decomposition
\[
\mathfrak{Der}(\mathfrak{h}_2({\mathbb K}))=\mathfrak{a}_2^\prime({\mathbb K})+C({\mathbb K}^\prime)+\mathfrak{Der}({\mathbb K}).
\]
holds at the vector space level. When pairing with $L(\mathfrak{h}_2({\mathbb K}))$ and after canceling the multiples of the identity, the reduced structure algebra of the Jordan algebra $\mathfrak{h}_2({\mathbb K})$ is given by
\begin{eqnarray*}
\mathfrak{Str}^\prime(\mathfrak{h}_2({\mathbb K}))&\cong&\mathfrak{a}_2^\prime({\mathbb K})\oplus C({\mathbb K}^\prime)\oplus\mathfrak{Der}({\mathbb K})\oplus\mathfrak{h}_2^\prime({\mathbb K})\\
&\cong&M_2^\prime({\mathbb K})\oplus C({\mathbb K}^\prime)\oplus\mathfrak{Der}({\mathbb K})
\end{eqnarray*}
as vector spaces.
\\

To see if this definition makes sense, compare it with the cases that are already known. For ${\mathbb K}={\mathbb R}, {\mathbb C}, {\mathbb H}$, $\mathfrak{Str}^\prime(\mathfrak{h}_2({\mathbb K}))$ is the same vector space as $\mathfrak{sl}(2,{\mathbb K})$. Hence, this is a good generalization of $\mathfrak{sl}(2,{\mathbb K})$ for the octonions. Moreover, $\mathfrak{Str}^\prime(\mathfrak{h}_2({\mathbb O}))$ as a vector space over the real numbers has dimension 45; the same as $\mathfrak{so}(9,1)$.
\\

The next step is constructing a bracket that will provide $\mathfrak{Str}^\prime(\mathfrak{h}_2({\mathbb O}))$ with a Lie algebra structure. To do this, identify the transformations in $\mathfrak{Str}^\prime(\mathfrak{h}_2({\mathbb O}))$ acting over matrices $X\in\mathfrak{h}_2({\mathbb O})$ with transformations in $\mathfrak{so}(9,1)$. Then, the bracket in $\mathfrak{Str}^\prime(\mathfrak{h}_2({\mathbb O}))$ will be given so that such correspondence is automatically a Lie algebra isomorphism. That is, the bracket in $\mathfrak{sl}(2,{\mathbb O})=M_2^\prime({\mathbb O})\oplus C({\mathbb O}^\prime)\oplus\mathfrak{Der}({\mathbb O})$ will be defined through
\[
\big[{\cal M},{\cal N}\big]=\varphi^{-1}\big(\varphi({\cal M})\varphi({\cal N})-\varphi({\cal N})\varphi({\cal M})\big)
\]
for all ${\cal M}, {\cal N}\in\mathfrak{sl}(2,{\mathbb O})$, where $\varphi\colon\mathfrak{sl}(2,{\mathbb O})\to\mathfrak{so}(9,1)$ is given by the following identifications.
\\

Let $N=\begin{pmatrix*}[r]a&b\\ c&-a\end{pmatrix*}\in M_2^\prime({\mathbb O})$ and $X=\begin{pmatrix}x_0+x_9&x\\ \bar{x}&x_0-x_9\end{pmatrix}\in\mathfrak{h}_2({\mathbb O})$ where $a$, $b$, and $c$ are expanded over the standard basis for the octonions and identify the vector $(x_0,x_1,\ldots,x_9)\in{\mathbb R}^{10}$ with the matrix $X$ as before. The transformation given by
\[
X\longmapsto N\,X+X\,N^\dagger
\]
corresponds to the following transformation in $\mathfrak{so}(9,1)$,
\[
\left(
\begin{smallmatrix}
0 & b_0+c_0 & b_1-c_1 & b_2-c_2 & b_3-c_3 & b_4-c_4 & b_5-c_5 & b_6-c_6 & b_7-c_7 & 2a_0 \\
b_0+c_0 & 0 & -2a_1 & -2a_2 & -2a_3 & -2a_4 & -2a_5 & -2a_6 & -2a_7 & c_0-b_0 \\
b_1-c_1 & 2a_1 & 0 & 0 & 0 & 0 & 0 & 0 & 0 & -(c_1+b_1) \\
b_2-c_2 & 2a_2 & 0 & 0 & 0 & 0 & 0 & 0 & 0 & -(c_2+b_2) \\
b_3-c_3 & 2a_3 & 0 & 0 & 0 & 0 & 0 & 0 & 0 & -(c_3+b_3) \\
b_4-c_4 & 2a_4 & 0 & 0 & 0 & 0 & 0 & 0 & 0 & -(c_4+b_4) \\
b_5-c_5 & 2a_5 & 0 & 0 & 0 & 0 & 0 & 0 & 0 & -(c_5+b_5) \\
b_6-c_6 & 2a_6 & 0 & 0 & 0 & 0 & 0 & 0 & 0 & -(c_6+b_6) \\
b_7-c_7 & 2a_7 & 0 & 0 & 0 & 0 & 0 & 0 & 0 & -(c_7+b_7) \\
2a_0 &b_0-c_0 & b_1+c_1 & b_2+c_2 & b_3+c_3 & b_4+c_4 & b_+c_5 & b_6+c_6 & b_7+c_7 & 0
\end{smallmatrix}
\right)
\]
which shall be the definition for $\varphi(N)$. Elements $C_d\in C({\mathbb O}^\prime)$ and $\mathfrak{g}\in\mathfrak{Der}({\mathbb O})$ transform $X$ acting over its entries. Writing $d=d_1\,e_1+\cdots+d_7\,e_7$, the transformation $C_d$ and $\mathfrak{g}$ correspond with
\[
2
\left(
\begin{smallmatrix}
0&0&\hspace*{.25cm}0&\hspace*{.25cm}0&\hspace*{.25cm}0&\hspace*{.25cm}0&\hspace*{.25cm}0&\hspace*{.25cm}0&\hspace*{.25cm}0&0\\
0&0&\hspace*{.25cm}0&\hspace*{.25cm}0&\hspace*{.25cm}0&\hspace*{.25cm}0&\hspace*{.25cm}0&\hspace*{.25cm}0&\hspace*{.25cm}0&0\\
0&0&\hspace*{.25cm}0&-d_7&-d_5&\hspace*{.25cm}d_6&\hspace*{.25cm}d_3&-d_4&\hspace*{.25cm}d_2&0\\
0&0&\hspace*{.25cm}d_7&\hspace*{.25cm}0&\hspace*{.25cm}d_6&\hspace*{.25cm}d_5&-d_4&-d_3&-d_1&0\\
0&0&\hspace*{.25cm}d_5&-d_6&\hspace*{.25cm}0&-d_7&-d_1&\hspace*{.25cm}d_2&\hspace*{.25cm}d_4&0\\
0&0&-d_6&-d_5&\hspace*{.25cm}d_7&\hspace*{.25cm}0&\hspace*{.25cm}d_2&\hspace*{.25cm}d_1&-d_3&0\\
0&0&-d_3&\hspace*{.25cm}d_4&\hspace*{.25cm}d_1&-d_2&\hspace*{.25cm}0&-d_7&\hspace*{.25cm}d_6&0\\
0&0&\hspace*{.25cm}d_4&\hspace*{.25cm}d_3&-d_2&-d_1&\hspace*{.25cm}d_7&\hspace*{.25cm}0&-d_5&0\\
0&0&-d_2&\hspace*{.25cm}d_1&-d_4&\hspace*{.25cm}d_3&-d_6&\hspace*{.25cm}d_5&\hspace*{.25cm}0&0\\
0&0&\hspace*{.25cm}0&\hspace*{.25cm}0&\hspace*{.25cm}0&\hspace*{.25cm}0&\hspace*{.25cm}0&\hspace*{.25cm}0&\hspace*{.25cm}0&0
\end{smallmatrix}
\right)
\qquad{\rm and}\qquad
\left(
\begin{smallmatrix}
0&0&0&0&0&0&0&0&0&0\\
0&0&0&0&0&0&0&0&0&0\\
0&0&&&&&&&&0\\
0&0&&&&&&&&0\\
0&0&&&&\uparrow&&&&0\\
0&0&&&\leftarrow&\mathfrak{g}&\rightarrow&&&0\\
0&0&&&&\downarrow&&&&0\\
0&0&&&&&&&&0\\
0&0&&&&&&&&0\\
0&0&0&0&0&0&0&0&0&0
\end{smallmatrix}
\right)
\]
in $\mathfrak{so}(9,1)$, where $\mathfrak{g}$ is identified with its matrix representation inside $\mathfrak{so}(7)$. These matrices determine $\varphi(C_d)$ and $\varphi(\mathfrak{g})$ respectively.
\\

Since it is intended to construct a Lie bracket on $\mathfrak{sl}(2,{\mathbb O})$ in order for it to be consistent with the Lie bracket in $\mathfrak{so}(9,1)$, the above linear identifications given by $\varphi$ provide the following inherited Lie algebra structure. Writing $M=\begin{pmatrix*}[r]w&y\\ z&-w\end{pmatrix*}$, for $M,N\in M_2^\prime({\mathbb O})$, $C_d,C_{d^\prime}\in C({\mathbb O}^\prime)$, and $\mathfrak{g},\mathfrak{g}^\prime\in\mathfrak{Der}({\mathbb O})$, the Lie bracket in $\mathfrak{sl}(2,{\mathbb O})$ is given by
\begin{eqnarray*}
\left[C_d,C_{d^\prime}\right]&=&C_d\,C_{d^\prime}-C_{d^\prime}\,C_d\\
\left[\mathfrak{g},\mathfrak{g}^\prime\right]&=&\mathfrak{g}\,\mathfrak{g}^\prime-\mathfrak{g}^\prime\,\mathfrak{g}\\
\left[\mathfrak{g},C_d\right]&=&C_{\mathfrak{g}(d)}\\
\left[C_d,M\right]&=&\begin{pmatrix*}[r]C_d(w)&C_d(y)\\ C_d(z)&-C_d(w)\end{pmatrix*}\\
\left[\mathfrak{g},M\right]&=&\begin{pmatrix*}[r]\mathfrak{g}(w)&\mathfrak{g}(y)\\ \mathfrak{g}(z)&-\mathfrak{g}(w)\end{pmatrix*}\\
\left[M,N\right]&=&\left(M\,N-N\,M-\frac{1}{2}{\rm Tr}(MN-NM)\,I\right)+C_{\frac{1}{6}{\rm Tr}(MN-NM)}+\mathfrak{g}_{M,N}
\end{eqnarray*}
where $\mathfrak{g}_{M,N}\in\mathfrak{Der}({\mathbb O})$ stands for $\mathfrak{g}_{M,N}:=\frac{2}{3}D_{w,a}+\frac{1}{3}D_{y,b}+\frac{1}{3}D_{z,c}$.
\\

The way this bracket has been constructed automatically guarantees the Jacobi identity, making $\mathfrak{sl}(2,{\mathbb O})$ a Lie algebra and $\varphi\colon\mathfrak{sl}(2,{\mathbb O})\to\mathfrak{so}(9,1)$ a Lie algebra isomorphism.
\section{The Lie group $SL(2,{\mathbb O})$}\label{groupSL2O}
The previous section suggests $SL(2,{\mathbb O})$ be constructed as a group of transformations acting over the vector space $\mathfrak{h}_2({\mathbb O})$. Despite lack of associativity on behalf of the octonions, composition of transformations is still an associative operation.
\\

Given $X=\begin{pmatrix}\alpha&x\\ \bar{x}&\beta\end{pmatrix}\in\mathfrak{h}_2({\mathbb O})$ and $M=\begin{pmatrix}a&b\\ c&d\end{pmatrix}\in M_2({\mathbb O})$, consider the transformation
\begin{equation}\label{def-trans-hermit}
X\longmapsto \frac{1}{2}\Big((MX)M^\dagger+M(XM^\dagger)\Big)=\phi_M(X),
\end{equation}
which can be explicitly given as
\[
{
\left(
\begin{smallmatrix}
\alpha\,||a||^2+2\,{\rm Re}(ax\bar{b})+\beta\,||b||^2
&
\alpha\,a\bar{c}+\beta\,b\bar{d}+\frac{1}{2}\left(a(x\bar{d})+(ax)\bar{d}+b(\bar{x}\bar{c})+(b\bar{x})\bar{c}\right)
\\
\alpha\,c\bar{a}+\beta\,d\bar{b}+\frac{1}{2}\left((d\bar{x})\bar{a}+d(\bar{x}\bar{a})+(cx)\bar{b}+c(x\bar{b})\right)
&
\alpha\,||c||^2+2\,{\rm Re}(cx\bar{d})+\beta\,||d||^2
\end{smallmatrix}
\right)
}_{.}
\]
The set of all transformations defined by~\eqref{def-trans-hermit} generate a free monoid. The subset of all invertible transformations form the largest group contained in the monoid. $SL(2,{\mathbb O})$ shall denote the subgroup of invertible transformations that preserve the determinant.
\\

Since $MM^\dagger\in\mathfrak{h}_2({\mathbb O})$ its determinant is well defined and given by
\[
{\rm det}(MM^\dagger)=||a||^2||d||^2+||b||^2||c||^2-\left(a\bar{c}\right)\left(d\bar{b}\right)-\left(b\bar{d}\right)\left(c\bar{a}\right).
\]
Straight calculation shows that
\[
{\rm det}\left(\phi_M(X)\right)={\rm det}(MM^\dagger)\,{\rm det}(X)
\]
if, and only if,
\begin{equation}\label{det-cond}
\begin{split}
\left((ax)\bar{d}+a(x\bar{d})+(b\bar{x})\bar{c}+b(\bar{x}\bar{c})\right)\left(d(\bar{x}\bar{a})+(d\bar{x})\bar{a}+c(x\bar{b})+(cx)\bar{b}\right)\quad\\
-\left(a(x\bar{b})+(b\bar{x})\bar{a}+(ax)\bar{b}+b(\bar{x}\bar{a})\right)\left(c(x\bar{d})+(d\bar{x})\bar{c}+(cx)\bar{d}+d(\bar{x}\bar{c})\right)
\end{split}
\end{equation}
is equal to $4\,||x||^2\,{\rm det}(MM^\dagger)$ for all $x\in{\mathbb O}$.
\\

The following Lemma will be useful for simplifying the previous expression.
\begin{lemma}\label{lemma4oct}
For any four octonions, $a,b,c,d$, it is always true that
\[
2\,{\rm Re}(ab)\,{\rm Re}(cd)={\rm Re}\left((a\bar{c})(\bar{d}b)+(ad)(cb)\right).
\]
\end{lemma}
\begin{proof}
The left hand side can be rewritten as
\[
2{\rm Re}(ab){\rm Re}(cd)\!=\!2{\rm Re}(ab){\rm Re}(dc)\!=\!{\rm Re}\left(a(dc+\bar{c}\bar{d})b\right)\!=\!{\rm Re}\left(a\big((dc)b+(\bar{c}\bar{d})b\big)\right)
\]
while the right hand side is ${\rm Re}\left((a\bar{c})(\bar{d}b)+(ad)(cb)\right)\!=\!{\rm Re}\left(a\big(\bar{c}(\bar{d}b)+d(cb)\big)\right).$ Since
\[
(dc)b-d(cb)+(\bar{c}\bar{d})b-\bar{c}(\bar{d}b)=\big[d,c,b\big]+\big[\bar{c},\bar{d},b\big]=\big[d,c,b\big]-\big[d,c,b\big]=0,
\]
then clearly ${\rm Re}\left(a\big((dc)b+(\bar{c}\bar{d})b\big)\right)-{\rm Re}\left(a\big(\bar{c}(\bar{d}b)+d(cb)\big)\right)=0$, and this finishes the proof.
\end{proof}
\begin{lemma}\label{detfactor}
The expression in~\eqref{det-cond} is equal to
\begin{equation*}
\begin{split}
2\left(||a||^2||d||^2+||b||^2||c||^2\right)||x||^2+2{\rm Re}\left(\big(a(x\bar{d})\big)\big(d(\bar{x}\bar{a})\big)\right)+2{\rm Re}\left(\big(b(\bar{x}\bar{c})\big)\big(c(x\bar{b})\big)\right)
\\
-4{\rm Re}\left((a\bar{c})(d\bar{x})(x\bar{b})+(ax)(\bar{x}\bar{c})(d\bar{b})\right)+\big[a,d,x\big]\big[b,c,x\big]+\big[b,c,x\big]\big[a,d,x\big].
\qquad\quad
\end{split}
\end{equation*}
\end{lemma}
\begin{proof}
The first line in~\eqref{det-cond} is equal to
\begin{equation}\label{firstline}
\begin{split}
2\left(||a||^2||d||^2+||b||^2||c||^2\right)||x||^2\qquad\qquad\qquad\qquad\qquad\qquad\qquad\qquad\;\\
+2{\rm Re}\left(\big(a(x\bar{d})\big)\big(d(\bar{x}\bar{a})\big)\right)+2{\rm Re}\left(\big(b(\bar{x}\bar{c})\big)\big(c(x\bar{b})\big)\right)\qquad\qquad\qquad\qquad\quad\;\\
+\left((ax)\bar{d}+a(x\bar{d})\right)\left(c(x\bar{b})+(cx)\bar{b}\right)+\left((b\bar{x})\bar{c}+b(\bar{x}\bar{c})\right)\left(d(\bar{x}\bar{a})+(d\bar{x})\bar{a}\right).
\end{split}
\end{equation}
Using the fact that ${\rm Re}\left(\left[a,x,\bar{b}\right]\right)=0$ and ${\rm Re}\left(\left[c,x,\bar{d}\right]\right)=0$, the second line in~\eqref{det-cond} is equal to
\begin{equation}\label{replacing}
-8\,{\rm Re}\left(a(x\bar{b})\right)\,{\rm Re}\left(c(x\bar{d})\right)
-8\,{\rm Re}\left((ax)\bar{b}\right)\,{\rm Re}\left((cx)\bar{d}\right).
\end{equation}
Applying the formula provided in Lemma~\ref{lemma4oct}, replacing $b$ and $d$ for $x\bar{b}$ and $x\bar{d}$ in the first term and replacing $a$ and $c$ for $ax$ and $cx$ in the second term, \eqref{replacing} can be rewritten as
\begin{equation}\label{secondline}
\begin{split}
-4\,{\rm Re}\left((a\bar{c})(d\bar{x})(x\bar{b})\right)-2\left((a(x\bar{d}))(c(x\bar{b}))+((b\bar{x})\bar{c})((d\bar{x})\bar{a})\right)\;\,\\
-4\,{\rm Re}\left((ax)(\bar{x}\bar{c})(d\bar{b})\right)-2\left(((ax)\bar{d})((cx)\bar{b})+(b(\bar{x}\bar{c}))(d(\bar{x}\bar{a}))\right).
\end{split}
\end{equation}
Adding the expressions from~\eqref{firstline} and~\eqref{secondline} yields the expected result.
\end{proof}
Close inspection of Lemma~\ref{detfactor} reveals that $\det(X^\prime)=\det(MM^\dagger)\det(X)$, where $M=\begin{pmatrix}a&b\\ c&d\end{pmatrix}$, only for the cases:
\begin{itemize}
\item
$a=0$ and $[b,c,x]=0$ $\forall x\in{\mathbb O}$ (analogously when $b=0$ and $[a,d,x]=0$ $\forall x\in{\mathbb O}$, $c=0$ and $[a,d,x]=0$ $\forall x\in{\mathbb O}$, $d=0$ and $[b,c,x]=0$ $\forall x\in{\mathbb O}$)
\item
$[u,v,x]=0$ $\forall u,v\in\{a,b,c,d\},\,\forall x\in{\mathbb O}$.
\end{itemize}
In either of these cases, it is true that $\det(MM^\dagger)=||ad-bc||^2$. Nevertheless, for these transformations to be invertible it is necessary for all octonions to share one imaginary direction (as in the second case). This is in complete agreement with the work of Manogue \& Schray, presented in~\cite{manogue-finite}, where they impose the same condition in order to guarantee the Hermitian action $X\mapsto MXM^\dagger$ is well defined for the octonionic case. This aspect, along with further discussions regarding when is ${\rm det}(MXM^\dagger)$ equal to ${\rm det}(MM^\dagger)\,{\rm det}(X)$ are also found in~\cite{manogue-moebius,manogue-spinor}.
\\

Requiring all four octonions to share their imaginary direction can be expressed in the following way. Consider eight real numbers, $\mu_i, \nu_i$ for indices $i=a,b,c,d$, and an octonion $q$ such that ${\rm Re}(q)=0$. Then the generating elements for $SL(2,{\mathbb O})$ are transformations $\phi_M$ where $M$ is of the form
\begin{equation}\label{genSL2O}
{
\begin{pmatrix}
\mu_a+\nu_a\,q&\quad&\mu_b+\nu_b\,q\\
\mu_c+\nu_c\,q&&\mu_d+\nu_d\,q\\
\end{pmatrix}
}_{.}
\end{equation}
These are the same transformations that are used in~\cite{manogue-finite,manogue-moebius,manogue-spinor} as elements in the generating set for $SL(2,{\mathbb O})$. The main result shall now be proven.
\begin{theorem}
The set of all invertible and determinant preserving transformations $\phi_M
\:\colon
{\mathfrak h}_2({\mathbb O})\to{\mathfrak h}_2({\mathbb O})$ defined as in~\eqref{def-trans-hermit} generate a free group with the operation given by composition. It is also a Lie group whose Lie algebra is precisely $\mathfrak{sl}(2,{\mathbb O})$.
\end{theorem}
\begin{proof}
Composition of transformations is clearly an associative operation. The identity element corresponds to the identity matrix; $a=d=1$ and $b=c=0$. If the transformation $\phi_M$ is a generating element, $M$ shall be of the form given in~\eqref{genSL2O}. Explicit calculation shows that in this case
\[
{\phi_M}^{-1}=\phi_{(ad-bc)^{-1}\,{\rm id}}\circ\phi_{{\rm adj}(M)}
\]
where $\circ$ represents composition of transformations, ${\rm id}$ is the identity matrix in $M_2({\mathbb O})$, and ${\rm adj}(M)$ is the adjoint\footnote{Given $M$ is $2\times 2$, its adjoint matrix is well defined even while having octonionic entries and ${\rm adj}(M)=\begin{pmatrix}d&-c\\ -c&a\end{pmatrix}$.} matrix of $M$. The inverse element is also determinant preserving, given the fact that $||ad-bc||^2=1$ for the generating transformation. The inverse element of a composition of generating transformations follows the usual rule. Therefore, we are in presence of a free group.
\\

To show that this set is a Lie group, simply view these transformations as a subgroup of $GL(10, {\mathbb R})$ and consider the inherited differential structure. It is clear that both the maps that send two transformations to their composition or an element onto its inverse are continuous. Consider a curve of generating transformations passing through the identity at $t=0$; that is, an element whose entries are continuous\footnote{Since the norm in ${\mathbb O}$ is the same as in ${\mathbb R}^8$, endow the octonions with the metric topology of an Euclidean space of eight components.} functions from ${\mathbb R}$ in ${\mathbb O}$ such that $a(0)=d(0)=1$ and $b(0)=c(0)=0$. Recalling the formula given in~\eqref{def-trans-hermit}, the tangent vector along the given curve at the identity transforms elements in $\mathfrak{h}_2({\mathbb O})$ in the exact same fashion as $\mathfrak{so}(9,1)\cong\mathfrak{sl}(2,{\mathbb O})$,
\[
X\longmapsto N\,X+X\,N^\dagger
\]
where $N={\left.\frac{{\rm d}}{{\rm d}t}M(t)\right|}_{t=0}$ and $M(t)$ is the matrix related to the transformation.
\\

The next step is to exhibit that for any vector in $\mathfrak{sl}(2,{\mathbb O})$ there is an element in the group such that the vector is tangent to it. Given a curve, that passes through the identity for $t=0$, determined by only one generating transformation, the tangent vector at the identity is
\[
\begin{pmatrix}
\frac{1}{2}\left(\dot{a}-\dot{d}\right)&\dot{b}\\
\dot{c}&\frac{1}{2}\left(\dot{d}-\dot{a}\right)
\end{pmatrix}
+C_{\frac{1}{2}\left(\dot{d}+\dot{a}\right)}
\]
where the dot represents the derivative respect to $t$ evaluated at $t=0$. Recall that the first element is in $M^\prime({\mathbb O})$ and the second in $C({\mathbb O}^\prime)$, which are traceless $2\times2$ octonionic matrices and commutator maps respectively. Notice that the vectors in $\mathfrak{g}_2$ cannot appear as tangent to curves determined by just one generating transformation. Writing $d=-a$, all the vectors in $M^\prime({\mathbb O})$ are reached. On the other hand, imposing $b=0$, $c=0$, and $d=a$, all the vectors in $C({\mathbb O}^\prime)$ are reached. For the $\mathfrak{g}_2$ part of the algebra, composition of two elements are enough to form a basis. Nonetheless, according to the description given in~\eqref{gut} and Lemma~\ref{tangenG2}, the following composition of transformations
\[
\phi_{(u(t))^{-1}\,{\rm id}}\circ\phi_{u(0)\,{\rm id}}\circ\phi_{u(t)\,{\rm id}}\circ\phi_{(u(0))^{-1}\,{\rm id}}
\]
reproduces $G_{a,b}^t\in G_2$ and therefore the tangent vector at the identity is exactly the canonical derivation $D_{a,b}$ described in~\eqref{Dab}. Hence, this set of invertible and determinant preserving transformations from $\mathfrak{h}_2({\mathbb O})$ onto itself generate a Lie group whose Lie algebra is $\mathfrak{sl}(2,{\mathbb O})$.
\end{proof}
\section*{Acknowledgements}
Special thanks are given to my advisor, Professor Alvaro Restuccia, who has followed this work very closely as part of my Doctoral dissertation. Also, I would like to express my gratitude to the referee for many valuable suggestions that have improved the presentation of this paper.

\end{document}